\documentclass{amsart}
\usepackage{hyperref}
\usepackage{amssymb}
\usepackage{amsaddr}
\usepackage{amsmath}
\usepackage{subcaption}
\usepackage{graphicx}
\usepackage{units}
\usepackage[margin=1.6in]{geometry}

\DeclareMathOperator{\Id}{Id}

\newtheorem{thm}{Theorem}[section]
\newtheorem{prop}[thm]{Proposition}
\newtheorem{cor}[thm]{Corollary}
\newtheorem{defi}[thm]{Definition}

\newtheorem{remark}[thm]{Remark}

\newcommand\xip{p}
\newcommand{\LM}[1]{\hbox{\vrule width.2pt \vbox to#1pt{\vfill \hrule width#1pt height.2pt}}}
\newcommand{\LL}{{\mathchoice{\,\LM7\,}{\,\LM7\,}{\,\LM5\,}{\,\LM{3.35}\,}}}

\newcommand\eps{\varepsilon}
\newcommand\Z{{\mathbb Z}}
\newcommand\R{{\mathbb R}}
\newcommand\dist{\mathop{{\rm dist}}}

\newcommand\calH{{\mathcal H}}

\newcommand\C{\mathbb C}

\title[ansatz-free data-driven inference]{Convergence rates for ansatz-free data-driven inference in physically constrained problems}
\author[S.~Conti, F.~Hoffmann and M.~Ortiz]{
S.~Conti${}^1$, F.~Hoffmann${}^1$ and M.~Ortiz${}^{2,3}$
}

\address
{${}^1$Institut f\"ur Angewandte Mathematik, Universit\"at Bonn, Germany \\
    ${}^2$Hausdorff Center for Mathematics, Universit\"at Bonn, Germany \\
    ${}^3$Division of Engineering and Applied Science, California Institute of Technology, Pasadena}

\begin{document}

\maketitle

\begin{abstract}
We study a Data-Driven approach to inference in physical systems in a measure-theoretic framework. The systems under consideration are characterized by two measures defined over the phase space: i) A physical likelihood measure expressing the likelihood that a state of the system be admissible, in the sense of satisfying all governing physical laws; ii) A material likelihood measure expressing the likelihood that a local state of the material be observed in the laboratory. We assume deterministic loading, which means that the first measure is supported on a linear subspace. We additionally assume that the second measure is only known approximately through a sequence of empirical (discrete) measures. We develop a method for the quantitative analysis of convergence based on the flat metric and obtain error bounds both for annealing and the discretization or sampling procedure, leading to the determination of appropriate quantitative annealing rates. Finally, we provide an example illustrating the application of the theory to transportation networks.
\end{abstract}

\section{Introduction}\label{sec:intro}

We consider the problem of inferring the probability of finding a physical system in a given state $z$ in a linear space $Z$, or phase space, which we assume to be finite-dimensional. For instance, if the system under consideration is an electrical circuit, then the state of the system consists of the array of potential differences across the elements of the circuit and the corresponding array of electric currents; if the system is a hydraulic network, then the state of the system consists of the array of head differences across each pipe and the corresponding array of mass fluxes; if the system is a mechanical truss structure, then the state of the system consists of the array of displacement differences, or strains, across each member and the corresponding array of internal forces, or stresses; {\sl et~cetera}. We note that, in all these examples, the state of the system consists of a pair of dual variables and the dimension of phase space is even.

Physical systems obey field equations, which place hard constraints on the possible states attainable by the system. These constraints are material independent and can be regarded as a restriction of the set of admissible states of the system. The view of field equations as constraints for purposes of analysis has a long-standing tradition in continuum mechanics and electromagnetism, and constitutes the foundation of recent methods of data-driven analysis \cite{kirchdoerfer2016data, conti2018data}, physically-informed neural networks (PINNs) \cite{Raissi:2017} and other applications of modern data science. Classically, deterministic problems in mathematical physics are closed by further restricting the states of the system to lie in a subset representing the material law of the system, i.~e., the locus of states attainable by a specific material.

\begin{figure}[ht]
\begin{center}
	\begin{subfigure}{0.35\textwidth}\caption{} \includegraphics[width=0.99\linewidth]{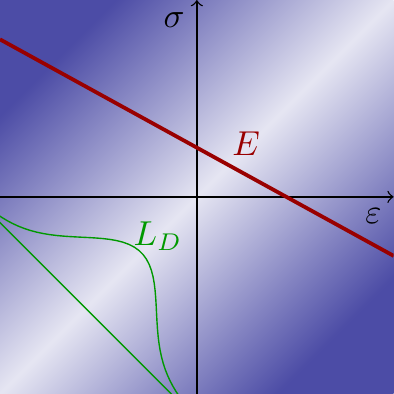}
	\end{subfigure}
    \quad\quad
	\begin{subfigure}{0.35\textwidth}\caption{} \includegraphics[width=0.99\linewidth]{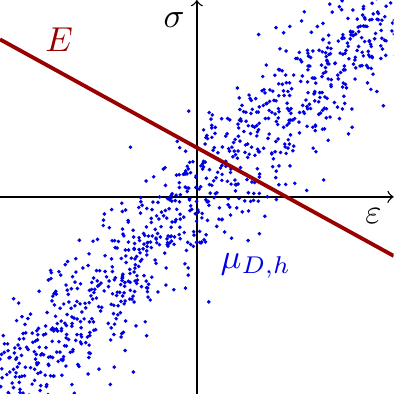}
	\end{subfigure} 
    \caption{Classical inference. a) Material likelihood function $L_D$, here in the form of a sliding Gaussian (dark: low likelihood; light: high likelihood), constraint set $E$ and likelihood function $L$ obtained by restricting $L_D$ to $E$. b) Empirical likelihood measure $\mu_{D,h}$ sampled from $L_D$.} \label{lRkk0p}
\end{center}
\end{figure}

In this paper, we work within a general framework 
\cite{Conti:2021} for systems in which the material law and the admissibility constraints are described by positive Radon likelihood measures $\mu_D \in \mathcal{M}(Z)$ and $\mu_E \in \mathcal{M}(Z)$, respectively, representing the likelihood of $y \in Z$ being a (local) material state observed in the laboratory and of $z \in Z$ being admissible. Before presenting our new contributions, the main ideas underlying the work may be summarized as follows. The admissible states of the system may be random, e.~g., due to the application of random forcing to the system. The observed material states may be random either because the material itself is random or because of experimental scatter, cf.~Fig.~\ref{lRkk0p}a. We expect the material states $y \in Z$ and admissible states $z \in Z$ of the system to be distributed according to a notion of intersection measure $\mu_D\cap\mu_E \in \mathcal{M}(Z\times Z)$, which can be qualitatively understood as the product measure $\mu_D\times\mu_E$ conditioned to $y=z$.
In the special case in which $\mu_D$ and $\mu_E$ are regular with respect to the Lebesgue measure, with continuous densities $L_D$ and $L_E$, the likelihood of finding the system at state $z \in Z$ is, simply, $L_D(z) L_E(z)$, which determines the intersection $\mu_D\cap\mu_E$. In particular, if $L_E(z) L_D(z)$ is integrable and non-zero, the expression
\begin{equation}\label{F5RerF}
    \mathbb{E}[f]
    =
    \frac
    {\int_Z  f(z) L_D(z) L_E(z) \, d\mathcal{L}^{2N}(z)}
    {\int_Z  L_D(z) L_E(z) \, d\mathcal{L}^{2N}(z)}
    \equiv
    \int_Z f(z) L(z) \, d\mathcal{L}^{2N}(z)
\end{equation}
gives the expected value of a quantity of interest $f \in C_c(Z)$. Similarly, if $\mu_D = L_D \mathcal{L}^{2N}$ with $L_D$ continuous and $\mu_E = \mathcal{H}^N \LL E$, corresponding to deterministic loading, then $\mu_D \cap \mu_E = L_D \mathcal{H}^N \LL E$, Fig.~\ref{lRkk0p}a. If $Z=\mathbb R^2$, $\mu_D=\mathcal H^1 \LL \mathbb R a$ and $\mu_E=\mathcal H^1 \LL \mathbb R b$, with $a$ and $b\in\mathbb R^2$ not parallel, then $\mu_D\cap \mu_E=\delta_0$. 

Suppose now that, as is often the case, the likelihood measure $\mu_D$ is not known exactly, but only approximately through sequences of empirical measures $(\mu_{D,h})$ obtained, e.~g., by means of material testing. Suppose further that the empirical measures supply an increasingly better approximation of $\mu_{D}$, e.~g., as a result of increasingly accurate and extensive measurements. We then may expect that, under appropriate conditions of convergence of $(\mu_{D,h})$, the sequence of approximate intersections $(\mu_{D,h}\cap \mu_{E})$ converge to the exact limiting likelihood measure $\mu_{D}\cap \mu_{E}$, thus defining a convergent approximation scheme for the inference problem.

A fundamental difficulty that arises immediately is that, for most notions of intersections of measures, the intersection of certain pairs of measures may not be well-defined or may be zero. Consider for example the setting in which both $\mu_D$ and $\mu_E$ are approximated by empirical measures $({\mu}_{D,h})$ and $({\mu}_{E,h})$. A conventional response to this challenge is to introduce Lebesgue-regular approximations $(\tilde{\mu}_{D,h})$ and $(\tilde{\mu}_{E,h})$ fitted to the data $(\mu_{D,h})$ and $(\mu_{E,h})$ by means of some method of regression. By regularity, $(\tilde{\mu}_{D,h})$ and $(\tilde{\mu}_{E,h})$ then have well-defined, {continuous} densities $\tilde{L}_{D,h}$ and $\tilde{L}_{E,h}$, respectively, and the intersections $(\tilde{\mu}_{D,h}\cap\tilde{\mu}_{E,h})$, which are intended as approximations of the exact intersection $\mu_D\cap\mu_E$, are simply given by $(\tilde{L}_{D,h}(z) \tilde{L}_{E,h}(z)){\mathcal L^{2N}(z)}$. However, there is no guarantee that this approximation will work in general and the approximations $(\tilde{\mu}_{D,h})$ and $(\tilde{\mu}_{E,h})$
need to be chosen appropriately. Here, we take a different approach using thermalizations.

Overall, there are three main cases of interest: 
\begin{enumerate}
    \item[(1)] Lebesgue-regular likelihoods;
    \item[(2)] empirical measures;
    \item[(3)] likelihood measures supported on linear subspaces.
\end{enumerate}
The general framework presented here allows for the physical likelihood and the material likelihood to be in either of these three classes independently of each other. For (2), we may consider that there is a sequence of approximating empirical measures with the limiting measure belonging to class (1) or (3). In this work, we focus on the case in which the physical likelihood $\mu_D$ is in (2) approximating (1), and the material likelihood $\mu_E$ is in (3), see~\eqref{muDh} and~\eqref{eqmuedetermintr} below. 

Whereas the measure-theoretical framework just outlined is remarkable for its directness and simplicity, a Bayesian reinterpretation of the rules of inference is often favored in the literature (cf.~\cite{Stuart:2010} and references therein). A common {\sl ansatz} is to introduce 
the representation $z=(\epsilon,\sigma)$ and a sequence of functions $g_{D,h}$, parameterized by a set of parameters $p_h$, providing the model
\begin{equation}\label{Y08FfO}
    \sigma = g_{D,h}(\epsilon; p_h) + \eta\,,
\end{equation}
where $\eta$ is a random variable, interpreted as observational noise, with likelihood $f_{D,h}(\,\cdot\,;q_h)$, parameterized by further parameters $q_h$, and to assume the approximate material likelihood to be of the form
\begin{equation}\label{49Qjt3}
    \tilde{L}_{D,h}((\epsilon,\sigma))
    =
    f_{D,h}\big(\sigma - g_{D,h}(\epsilon; p_h); q_h\big)\,.
\end{equation}
Evidently, if $f_{D,h}$ attains its maximum at $0$, then  (\ref{Y08FfO}) represents the most likely material law given the {\sl ansatz} and may thus be regarded as an identified, or learned, material model. A common choice for $g_{D,h}$ are neural networks, in the context of machine learning \cite{Herrmann:2020}, whereas a common choice of $f_{D,h}$ is Gaussian \cite{Knapik:2011, Dashti:2017}.
Common methods of regression used to determine the parameters from the data include classical methods of statistical inference such as maximum likelihood \cite{Knapik:2016, Dunlop:2020}, variational approaches based on the introduction of a loss function \cite{Stuart:2010}, measure-theoretical approaches based, e.~g., on the Wasserstein distance \cite{Bernton:2019} or the Kullback-Leibler discrepancy \cite{Pinski:2015}.

An essential problem with this approach is that the choice of material models $g_{D,h}$, observational noise $f_{D,h}$, priors, loss functions and parameterizations thereof are often not prescribed by theory or fundamental considerations but instead dictated by convenience. Worse still, the form of $f_{D,h}$ is often fixed throughout the sequence, e.~g., to be Gaussian, which renders the approximation scheme non-convergent in cases where the underlying likelihood measures $\mu_D$ and $\mu_E$ are not of the same form. Since the limiting likelihood measures $\mu_D$ and $\mu_E$ are often not known in practice, it is generally not possible to ensure that approximation schemes tied to particular choices of models and priors be convergent. In addition, it is clear that, even in the best of circumstances, representations of the form (\ref{Y08FfO}) and (\ref{49Qjt3}) introduce modeling bias and error and incur in loss of information relative to the data sets themselves.

The {\sl ansatz}-free approach of \cite{Conti:2021} adopted here leads to a direct connection between data and inference and is therefore lossless and free of modeling bias. In addition, it allows to treat unbounded likelihoods, a setting where it is not clear how to set-up a Bayesian framework that is able to address the questions of inference and approximation. Our approach overcomes the problem of unbounded likelihoods and zero intersection between the approximating likelihood measures $(\mu_{D,h})$ with $(\mu_{E})$ by recourse to thermalization and annealing. Specifically, we consider a sequence $\beta_h\to +\infty$ of reciprocal temperatures for $h\to+\infty$, and replace $\mu_h = \mu_{D,h}\times \mu_{E}$ by its thermalization
\begin{equation}\label{1AdJRO}
    \mu_{h,\beta_h}
    :=
    B_{\beta_h}^{-1} {\rm e}^{-\beta_h \|y-z\|^2}
    \mu_{h} ,
    \quad
    B_{\beta_h}
    :=
    \int_{Z}
        {\rm e}^{-\beta_h \|\xi\|^2}
    \, d\mathcal L^{{2N}}(\xi) .
\end{equation}
As $h \to \infty$, this regularization increasingly concentrates $\mu_h$ to the diagonal ${\rm diag}(Z\times Z)$ and is therefore expected to deliver the sought intersection $\mu_D\cap \mu_E$ in the limit. Suppose, for instance, that the approximate material likelihood measure is
\begin{equation}\label{muDh}
    \mu_{D,h}
    =
    \sum_{\xip\in P_h} m_\xip \delta_\xip ,
\end{equation}
where $(P_h)$ are point data sets in $Z$ and $m_\xip^{h} \ge 0$ are weights. Suppose, in addition, that the loading is deterministic,
 \begin{equation}\label{eqmuedetermintr}
    \mu_E = \mathcal{H}^N \LL E ,
\end{equation}
where $E$ is an affine subspace of $Z$ of dimension $N$ and $\mathcal{H}^N \LL E$ is the Hausdorff measure restricted to $E$. Then, the approximate expectation of a quantity $f \in C_b(Z\times Z)$ corresponding to (\ref{1AdJRO}) is, cf.~Section~\ref{h9Mes5},
\begin{align}\label{w9dT8P}
    \mathbb{E}_h[f]
    &=
    \frac
    {
        \sum_{\xip\in P_h}
            \int_E
                m_\xip B_{\beta_h}^{-1}
                {\rm e}^{-\beta_h \|\xip -z\|^2}
                f(\xip,z)
            \, d\calH^N(z)
    }
    {
        \sum_{\xip\in P_h}
            \int_E
                m_\xip B_{\beta_h}^{-1}
                {\rm e}^{-\beta_h \|\xip -z\|^2}
            \, d\calH^N(z)
    } ,
\end{align}
 which is explicit in the data and eschews the need for {\sl ans\"atze} of any type, be they material models or priors. We do not consider the definition of the constraint set $E$ as a modelling step given that it encodes the governing physical laws. Data-driven inference rules such as (\ref{w9dT8P}) are amenable to efficient numerical implementation in combination with stochastic quadrature formulas for the evaluation of the integrals \cite{Prume:2022}.

However, the analysis of \cite{Conti:2021}, based on the concept of weak convergence, is not quantitative. It does not permit to obtain convergence rates, both for the convergence of $\mu_\beta$ to some limit $\mu_\infty$ and for the convergence of $\mu_{h,\beta}$ to $\mu_\beta$. In particular, it is not clear the thermalization parameter $\beta_h$ in \eqref{w9dT8P} should be chosen in practice.

\subsection{Main Results}
\label{secmainresults}
Our aim is to obtain quantitative estimates for the convergence of $\mu_\beta$ to its limit $\mu_\infty$ and the convergence of $\mu_{h,\beta}$ to $\mu_\beta$, leading in particular to a prescription for the choice $\beta_h$ which ensures the desired convergence $\mu_{h,\beta_h} \to \mu_\infty$. In order to make  convergence quantitative, we work in a metric setting and not only in terms of weak convergence as in \cite{Conti:2021}. Our starting point is this observation:
\begin{itemize}
 \item The flat norm metrizes weak convergence on bounded and tight sets of measures.
\end{itemize}
We adopt, in this metric setting, the notion of transversality  and of diagonal concentration for (possibly unbounded) measures via thermalization as developed in \cite{Conti:2021}. We  introduce a weaker concept, \emph{weak transversality}, which corresponds to transversality along subsequences, see Definition~\ref{def:diag}, and circumvents the need for regularity assumptions on the measure $\mu$. Then, Prokhorov's theorem shows that:
\begin{itemize}
    \item If $\mu$ is such that the measures $\mu_\beta$ are uniformly bounded and uniformly tight, then it is weakly transversal and, in particular, it has one or more diagonal concentrations.
\end{itemize}
Having framed the thermalization problem in a  metric setting, we may make use of standard devices such as uniform convergence and diagonal subsequences. This permits to decouple the thermalization problem ($\beta\to\infty$) from the approximation problem ($h\to\infty$). A typical statement (not detailed here) is:
\begin{itemize}
 \item 
Assume that i) $\mu_h$ are uniformly transversal and ii) $\mu_h$ have thermalizations that are uniformly bounded, uniformly tight and uniformly approximate the thermalization of $\mu$. Then $\mu$ is transversal and its diagonal concentration is approximated by the diagonal concentrations of $\mu_h$.
\end{itemize}
In particular, the diagonal concentration of an unknown measure $\mu$ is recovered from the diagonal concentrations of an approximating sequence of measures. 

We demonstrate the usefulness of this abstract framework by considering specific classes of measures. We first focus on sub-Gaussian material likelihoods combined with a deterministic physical likelihoods. Specifically, we consider measures of the form
\begin{equation}\label{subgauss}
\mu={\rm e}^{-\Phi} \mathcal{L}^{2N} \times \mathcal{H}^{N}\LL E
\end{equation}
for some $N$-dimensional affine subspace $E$ of $Z$ and $\Phi:Z\to\R$ satisfying
\begin{equation}\label{trans-cond} 
    \beta_0 \| y - z \|^2
    +
    \Phi(y)
    \geq 
    c
    \big(
        \| y \|^2 + \| z \|^2
    \big)
    -
    b 
    \hskip5mm
    {\text{ for all $y\in Z$, $z\in E$}} 
\end{equation}
for some constants $\beta_0 > 0$, $c > 0$ and $b > 0$. Condition~\eqref{trans-cond} can be interpreted as a transversality condition in view of the following result.

\begin{thm}[informal, see Prop.~\ref{p5qmD1}
and Prop.~\ref{VkH97k}]\label{thm:therm}
Measures $\mu$ satisfying \eqref{subgauss}-\eqref{trans-cond} are weakly transversal and admit a diagonal concentration. Further, if $\Phi\in C^1$ and its derivative does not grow too fast,
then $\mu$ is strongly transversal and the thermalizations $\mu_\beta$ of $\mu$ converge to the diagonal concentration $\mu_\infty$ with rate $\beta^{-1/2}$.
\end{thm}

We remark that, for fully deterministic systems, the potential $\Phi$ is the indicator function of a set $D \subset Z$, in which case \eqref{trans-cond} corresponds precisely to the definition of transversality introduced in \cite{conti2018data}. 

Next, we consider the case in which the material likelihood measure $\mu_D$ is known only approximately through a sequence of empirical measures. We assume control both on the spatial resolution of the approximation (quantified by $\epsilon_h$) and on the weights (quantified by $\delta_h$).

\begin{thm}[informal, see Cor.~\ref{y1TP50}]\label{thm:approx}
Assume that for each $h>0$ we can find a partition $Z=\cup_{p\in P_h} A_p$, a parameter $\epsilon_h>0$ that captures how well the discretization approximates $\mu_D$ and a parameter $\delta_h>0$ that captures how well $m_p$ approximates the mass $\mu_D(A_p)$.
If $\epsilon_h,\delta_h \to 0$ and $\beta_h\to \infty$ as $h\to \infty$ in such a way that $\beta_h\epsilon_h^2$ is bounded, then 
\begin{equation}\label{therm-approx}
    \| \mu_{h,\beta_h} - \mu_{\beta_h} \|_{\rm FN}
    \leq C
(\delta_h + \beta_h^{1/2}\epsilon_h)
\sup_{\beta\ge\beta_0} \|\mu_\beta\|_{\rm TV}\,,
\end{equation}
provided $\mu$ has uniformly bounded and uniformly tight thermalizations.
\end{thm}

In practice, the parameter $\epsilon_h$ is determined from the experimental data sets. Theorem~\ref{thm:approx} shows how to choose the cooling schedule $\beta_h$ depending on $\epsilon_h$. For the case of sub-Gaussian measures of the form \eqref{subgauss}-\eqref{trans-cond}, we can make this analysis quantitative. Choosing $m_p=\mu_D(A_p)$, or $\delta_h=0$, Th.~\ref{thm:therm} and Th.~\ref{thm:approx} give
\begin{align}
\label{eqintrooptbeta} 
    \| \mu_{h,\beta_h} - \mu_\infty \|_{\rm FN} 
    &\le  \| \mu_{\beta_h} - \mu_\infty \|_{\rm FN}  + 
    \| \mu_{h,\beta_h} - \mu_{\beta_h} \|_{\rm FN}
    \leq 
    \frac{C_1}{\beta_h^{1/2}} 
    +
    C_2 \beta_h^{1/2} \epsilon_h ,
\end{align}
for some constants $C_1,C_2>0$ and $\beta_h$ sufficiently large.
The optimal cooling schedule is therefore $ \beta_h \sim \epsilon_h^{-1}$ and the convergence rate in the flat norm is $\sim\beta_h^{-1/2}$. 

In summary, the analysis in this work ensures convergence---at a well-defined rate---of point-data approximations, should said point data $(P_h)$ be obtained by means of well-designed experiments supplying an increasingly faithful representation of the underlying likelihood measure $\mu_D$. 
In practice however, the underlying  material likelihood measure $\mu_D$ may be unknown and the assumptions in the rigorous statement of Theorem~\ref{thm:approx} cannot be verified directly from the data. This epistemic bottleneck notwithstanding, we note that the convergence of the thermalized approximations $\mu_{h,\beta_h}$ can in principle be assessed by determining if the sequence is Cauchy, since that criterion does not require explicit knowledge of the limit. A detailed assessment of matters of implementation is beyond the scope of the present paper and may be found in \cite{Prume:2022}. 

Background material and the general problem formulation are introduced in Section~\ref{sec:problemstatement}. The abstract strategy for solving inference problems via thermalization and formulating approximations thereof is enunciated in Section~\ref{U0KYmj}. In Section~\ref{kA7NJU} we address approximation and convergence of material point-data sets. Finally, in Section~\ref{Waxd2V} we describe a representative example. 

\section{Problem formulation}\label{sec:problemstatement}

\subsection{Prolegomena}\label{1kFcm9}

We denote by $\mathcal{M}(\mathbb{R}^n)$ the set of {real-valued} Radon measures, {write $|\mu|$ for the total variation of the measure $\mu$, and sometimes shorten} $\|\mu\|_{\mathrm {TV}}:=|\mu|(\mathbb{R}^n)$. 
In addition, we denote by $\mathcal{M}_b(\mathbb{R}^n) := \{\mu\in\mathcal{M}(\mathbb{R}^n): |\mu|(\mathbb{R}^n) < \infty\}$ the set of bounded {real-valued} Radon measures.  {We denote by $\mathcal{M}^+(\mathbb{R}^n)$ and $\mathcal{M}^+_b(\mathbb{R}^n)$, respectively, the corresponding sets of non-negative measures.} We say that $\mu_k\rightharpoonup\mu$ in $\mathcal{M}(\mathbb{R}^n)$ in the wide {(or weak-*)} topology if
\begin{equation}\label{eqdefweakc0}
    \lim_{k\to\infty}\int_{\mathbb{R}^n} f d\mu_k
    =
    \int_{\mathbb{R}^n} f d\mu \hskip1cm
    \text{ for all } f\in C_c(\mathbb{R}^n) ,
\end{equation}
with $C_c(\mathbb{R}^n)$ the set of continuous functions over $\mathbb{R}^n$ with compact support. In addition, we say that $\mu_k\rightharpoonup\mu$ in $\mathcal{M}_b(\mathbb{R}^n)$  {in the narrow topology} if
\begin{equation}\label{eqdefweakbounded}
    \lim_{k\to\infty}\int_{\mathbb{R}^n} f d\mu_k
    =
    \int_{\mathbb{R}^n} f d\mu \hskip1cm
    \text{ for all } f\in C_b(\mathbb{R}^n),
\end{equation}
with $C_b(\mathbb{R}^n)$ the set of bounded continuous functions.

We remark that (\ref{eqdefweakbounded}) is the narrow convergence normally used for probability measures, which is stronger than weak convergence of finite Radon measures usually defined testing with $f\in C_0(\mathbb{R}^n)$. Condition \eqref{eqdefweakbounded} in particular (testing with $f=1$) implies $\mu_k(\R^n)\to\mu(\R^n)$. Indeed, it is possible to show that, for nonnegative measures, condition \eqref{eqdefweakbounded} is equivalent to \eqref{eqdefweakc0} and $\mu_k(\R^n)\to \mu(\R^n)<\infty$. 

We say that a subset $M\subseteq\mathcal M_b(\mathbb{R}^n)$ is uniformly tight if for every $\epsilon>0$ there is a compact set $K_\epsilon$ such that, for all $\mu\in M$, ${|\mu|}(\mathbb{R}^n \setminus K_\epsilon)<\epsilon$. The subset $M$ is uniformly bounded if there is $C>0$ such that, for all $\mu\in M$, ${|\mu|}(\mathbb{R}^n)<C$. By Prokhorov's theorem, sets that are uniformly tight and uniformly bounded are sequentially precompact with respect to the {narrow} convergence in $\mathcal M_b(\mathbb{R}^n)$.

The space $BL(\mathbb{R}^n)$ of bounded Lipschitz functions is a Banach space with norm
\begin{equation}
    \|f\|_{\rm BL} := \max\{\| f \|_\infty,\, \ell\, {\rm Lip}(f) \},
\end{equation}
where we introduce a fixed dimensional constant $\ell>0$. We recall that the narrow topology of $\mathcal{M}_b(\mathbb{R}^n)$ {as defined in \eqref{eqdefweakbounded}} is metrized by the flat norm
\begin{equation}\label{BDrq89}
    \| \mu \|_{\rm FN}
    :=
    \sup
    \Big\{
         \int_{\mathbb{R}^n} f \, d\mu  \, : \,
        \|f\|_{\rm BL} \leq 1
    \Big\} ,
\end{equation}
on uniformly tight and uniformly bounded subsets (cf., e.~g., \cite[Th.~2.7(i)]{Gwiazda:2010}). 

\subsection{Classical inference} \label{Ef19yX}

We consider systems whose state is represented by points in a finite-dimensional space $Z$, or phase space. It is equipped with a scalar product and corresponding norm $\|\cdot\|$ and is identified with $\mathbb{R}^{2N}$. The systems are characterized by two measures over $Z$: i) a physical likelihood measure $\mu_E$ expressing the likelihood that a state of the system be admissible; ii) a material likelihood measure $\mu_D$ expressing the likelihood that a local state of the material be observed in the laboratory (we ignore here the subtle question of the characterization of the random experimental setup which could lead to those observations).
In this paper, we focus on the case of deterministic loading, in the sense of \eqref{eqmuedetermintr}, and consider the measure $\mu = \mu_D \times \mu_E = \mu_D\times \calH^N\LL E$ on $Z\times Z$, which fully characterizes the system both as regards material behavior and constraints. The classical inference problem is then to determine the likelihood of observing a material state $y \in Z$ and an admissible state $z \in Z$ conditioned to $y=z$. As a result, we expect the likelihood measure solving the inference problem to be a certain concentration of $\mu$ to the diagonal ${\rm diag}(Z \times Z)$.  

Whereas neither $\mu_D$ nor $\mu_E$ are finite in general, we expect the intersection $\mu_D \cap \mu_E$ as defined in this work to be finite and non-degenerate in the cases of interest, i.~e., $0<\|\mu_D \cap \mu_E\|_{\rm TV} < +\infty$, cf.~Fig.~\ref{lRkk0p}a. In particular, $\mu_D \cap \mu_E$ can be normalized by $\|\mu_D \cap \mu_E\|_{\rm TV}$ to define a probability measure characterizing the expectation of outcomes of the system. The condition that the intersection $\mu_D \cap \mu_E$ be well-defined, finite and non-degenerate sets forth a general notion of {\sl transversality} between the measures $\mu_D$ and $\mu_E$.

\section{Thermalization}\label{U0KYmj}

We introduce a notion of thermalization and diagonal concentration that extends the one 
proposed in \cite{Conti:2021}.

\subsection{General setting}
Starting from a general measure $\mu \in \mathcal{M}^+(Z\times Z$), for $\beta>0$ we define the {\sl thermalized} measure
\begin{equation}\label{3pSTen}
    \mu_\beta
    :=
    w_\beta \, \mu ,
    \quad
    w_\beta(y,z)
    :=
    B_\beta^{-1} {\rm e}^{-\beta \|y-z\|^2} ,
    \quad
      B_{\beta}
    :=
    \int_{Z}
        {\rm e}^{-\beta \|\xi\|^2}
    \, d\mathcal L^{{2N}}(\xi)
    =\beta^{-N}B_1    .
\end{equation}
Evidently, the weight $w_\beta$ suppresses the contributions away from the diagonal as $\beta \to \infty$, which suggests identifying all possible diagonal concentrations of $\mu$ with the weak limits of $\mu_\beta$, if any. We formalize this identification as follows.

\begin{defi}[Diagonal concentration, transversality]\label{def:diag}
Let $\mu \in \mathcal{M}^+(Z\times Z)$. We say that $\mu_\infty \in \mathcal{M}_b^+(Z\times Z)$ is a \emph{diagonal concentration} of $\mu$ if there is a sequence $\beta_h \to +\infty$ such that  $\mu_{\beta_h} \in \mathcal{M}_b^+(Z\times Z)$ for all $h$ and $\mu_{\beta_h} \rightharpoonup \mu_\infty$ in the narrow topology. We say that $\mu$ is \emph{weakly transversal} if it has a diagonal concentration.

We say that $\mu_\infty$ is a 
\emph{strong diagonal concentration}, and $\mu$ is 
\emph{strongly transversal} (or simply \emph{transversal}) if 
$\mu_{\beta} \in \mathcal{M}_b^+(Z\times Z)$ for all $\beta$ sufficiently large and $\mu_{\beta} \rightharpoonup \mu_\infty$. 

If $\mu,\nu \in \mathcal{M}^+(Z)$ are such that
$\mu\times \nu$ is strongly transversal, then we say that its (strong) diagonal concentration is the \emph{intersection} of $\mu$ and $\nu$, denoted by $\mu\cap \nu\in \mathcal{M}_b^+(Z\times Z)$.
\hfill$\square$
\end{defi}

The notion of (strong) transversality was introduced in  \cite[Def.~4.1]{Conti:2021}; the weaker notion of weak transversality is new. Both are different from the use of the term ``transversal'' in \cite{conti2018data}. We note that a measure $\mu \in \mathcal{M}^+(Z\times Z)$ can have several diagonal concentrations, or none at all. In the former case, the inference problem set forth by $\mu_D\times \mu_E$ has multiple solutions, whereas in the latter it has no solution. 
The theoretical background recalled in Section~\ref{1kFcm9} leads to some noteworthy consequences:
\begin{remark} ${}$ 
\begin{enumerate}
    \item  If the sequence $\mu_{\beta_h}$ (or $\mu_\beta$) is uniformly bounded and uniformly tight, then the narrow convergence can be equivalently replaced by strong convergence in the flat norm.
    \item By Prokhorov's theorem, if a sequence $\mu_{\beta_h}$ is uniformly bounded and uniformly tight then it has a subsequence which converges narrowly and $\mu$ is weakly transversal.
    \item If a diagonal concentration exists, then it is supported on ${\rm diag}(Z\times Z)$ (see \cite[Lemma~4.3]{Conti:2021}).
    \item If $\mu_\beta$ is uniformly bounded, uniformly tight and Cauchy with respect to the flat norm, then it is strongly transversal.
\end{enumerate}
\end{remark}
These observations evince that the property that $\mu_\beta$ be uniformly bounded and uniformly tight plays a crucial role in analysis.

\subsection{Sub-Gaussian likelihood measures}\label{81jQY7}
We discuss transversality in further detail in the case that $\mu_D$ is sub-Gaussian, in the sense made precise below.

\begin{prop}[Transversal likelihood functions, existence]\label{p5qmD1}
{Suppose that 
\begin{equation}\label{eqdefmusubg}
\mu={\rm e}^{-\Phi} \mathcal{L}^{2N} \times \mathcal{H}^{N}\LL E ,
\end{equation}
for some Borel function $\Phi:Z\to\R$
and some $N$-dimensional affine subspace $E$ of $Z$.}
Assume that  there exist constants $\beta_0 > 0$, $c > 0$ and $b > 0$ such that
\begin{equation}\label{9z2bLs} 
    \beta_0 \| y - z \|^2
    +
    \Phi(y)
    \geq 
    c
    \big(
        \| y \|^2 + \| z \|^2
    \big)
    -
    b 
    \hskip5mm
    {\text{ for all $y\in Z$, $z\in E$}} 
    .
\end{equation}
Then, the set $M := \{\mu_\beta,\ \beta\geq 2\beta_0\}$  is uniformly bounded  and uniformly tight in $\mathcal{M}_b(Z\times Z)$. In particular, $\mu$ is weakly transversal and has a diagonal concentration.
\end{prop}

\begin{proof}
We denote by $E_0$ the linear space parallel to $E$, by $E_0^\perp$ its orthogonal complement, and by $P_E$, $P_{E_0}$, $P_{E_0^\perp}$ the orthogonal projections with respect to the scalar product of $Z$. We change variables according to
\begin{equation}\label{eqchangvar}
 (y,z)=(\xi+\eta,\xi-\zeta) ,
 \hskip5mm
\text{ with $\xi\in E$, $\eta\in E_0^\perp$, $\zeta\in E_0$.}
\end{equation}
Specifically, $\xi=P_{E}y$, $\eta=y-\xi\in E_0^\perp$, $\zeta=\xi-z\in E_0$. Let $e_0:=P_E0\in E_0^\perp$ be the point of $E$ closest to the origin, so that $\xi-e_0\in E_0$. By orthogonality
\begin{equation}\begin{split}
 \|y\|^2&=\|\xi-e_0+\eta+e_0\|^2=
 \|\xi-e_0\|^2+\|\eta+e_0\|^2,\\
 \|z\|^2&=\|\xi-e_0-\zeta+e_0\|^2=
 \|\xi-e_0-\zeta\|^2+\|e_0\|^2,
 \end{split}
\end{equation}
which implies
\begin{equation}\label{eq:y2z2}
 \|y\|^2+\|z\|^2 \ge \frac14 \|\xi\|^2\,.
\end{equation}
Indeed, if $\|\xi\|\le 2\|e_0\|$ then
 $4\|z\|^2\ge4\|e_0\|^2\ge \|\xi\|^2$. Otherwise, $\|\xi\|\le \|\xi-e_0\|+\|e_0\|\le \|\xi-e_0\|+\frac12\|\xi\|$, and so
$4\|y\|^2\ge 4\|\xi-e_0\|^2\ge \|\xi\|^2$. This concludes the proof of \eqref{eq:y2z2}.
With $\|y-z\|^2=\|\eta+\zeta\|^2=\|\eta\|^2+\|\zeta\|^2$,
condition \eqref{9z2bLs} implies 
\begin{equation}\label{9z2bLsn}
    \beta_0 \left( \|\eta\|^2+\|\zeta\|^2\right)
    +
    \Phi(\xi+\eta)
    \geq
    \frac14 c
        \| \xi \|^2
    -
    b 
    \hskip3mm
    \text{ for all $\xi\in E,\eta\in E_0^\perp, \zeta \in E_0$}
    .
\end{equation}
For every $f\in C_b(Z\times Z;[0,\infty))$, from \eqref{eqdefmusubg} and \eqref{9z2bLsn} one obtains, for $\beta\ge 2\beta_0$,
\begin{equation}\label{eqmubsubgchv}
\begin{split}
    \int_{Z\times Z} f d\mu_\beta
    & = 
    B_\beta^{-1}
    \int_{E\times E_0^\perp\times E_0}
    f(\xi+\eta,\xi-\zeta)
    {\rm e}^{-\beta(\|\eta\|^2+\|\zeta\|^2)-\Phi(\xi+\eta) }d\calH^{3N}(\xi,\eta,\zeta)
    \\ & \le 
    B_\beta^{-1} 
    \int_{E\times E_0^\perp\times E_0}
    f(\xi+\eta,\xi-\zeta)
    {\rm e}^{-(\beta-\beta_0)(\|\eta\|^2+\|\zeta\|^2)+b-\frac c4\|\xi\|^2}
    d\calH^{3N}(\xi,\eta,\zeta)
    \\ & \le 
    B_\beta^{-1}
    \int_{E\times E_0^\perp\times E_0}
    f(\xi+\eta,\xi-\zeta)
    {\rm e}^{-\frac12 \beta(\|\eta\|^2+\|\zeta\|^2)+b-\frac c4\|\xi\|^2}
    d\calH^{3N}(\xi,\eta,\zeta).
\end{split}
\end{equation}
Inserting $f=1$ in \eqref{eqmubsubgchv}, we have, by Fubini's theorem,
\begin{equation}\label{o4KYs9}
\begin{split}
 \mu_\beta (Z\times Z)
     \le B_\beta^{-1}
 \int_{E_0^\perp\times E_0}
 {\rm e}^{-\frac12 \beta(\|\eta\|^2+\|\zeta\|^2)}
 d\mathcal L^{2N}(\eta,\zeta)
  \int_{E}
 {\rm e}^{b-\frac c4\|\xi\|^2}
 d\calH^{N}(\xi)
 \le C ,
\end{split}
\end{equation}
independently of $\beta$ (for every $\beta \ge 2\beta_0$),  and uniform boundedness follows.

For $R > 0$, let $A_R' := \{(y,z):\ \|\xi\|> R\}$. The same computation shows that
\begin{equation}\label{o4KYs9b}
\begin{split}
 \mu_\beta (A_R')
&     \le B_\beta^{-1}
 \int_{E_0^\perp\times E_0}
 {\rm e}^{-\frac12 \beta(\|\eta\|^2+\|\zeta\|^2)}
 d\mathcal L^{2N}(\eta,\zeta)
  \int_{E\cap \{\|\xi\|>R\}}
 {\rm e}^{b-\frac c4\|\xi\|^2}
 d\calH^{N}(\xi).
\end{split}
\end{equation}
The first integral equals $B_{\beta/2}=2^{-N}B_\beta$ and the second converges to zero as $R\to\infty$, independently of $\beta$. Therefore $\lim_{R\to\infty}  \sup_{\beta\ge2\beta_0}  \mu_\beta (A_R')=0$. Analogously, with $A_R'' := \{(y,z):\ \|\eta+\zeta\|> R\}$, changing variables to $\hat y:=\beta^{1/2}(\eta+\zeta)$,
\begin{equation}\label{o4KYs9c}
\begin{split}
    \mu_\beta (A_R'')
    & \le 
    B_\beta^{-1}
    \int_{(E_0^\perp\times E_0)\cap \{\|\eta+\zeta\|> R\} }
    {\rm e}^{-\frac12 \beta(\|\eta\|^2+\|\zeta\|^2)}
    d\mathcal L^{2N}(\eta,\zeta) 
    \int_{E}
    {\rm e}^{b-c\|\xi\|^2}
    d\calH^{N}(\xi)
    \\ & \le 
    C
    \int_{Z\cap \{\|\hat y\|> \beta^{1/2} R\} }
    {\rm e}^{-\frac12 \|\hat y\|^2}
    d\mathcal L^{2N}(\hat y)
    \le C
    \int_{Z\cap \{\|\hat y\|> \beta_0^{1/2} R\} }
    {\rm e}^{-\frac12 \|\hat y\|^2}
    d\mathcal L^{2N}(\hat y),
\end{split}
\end{equation}
which converges to zero as $R\to\infty$, implying $\lim_{R\to\infty}  \sup_{\beta\ge2\beta_0} \mu_\beta (A_R'')=0$. As
\begin{equation}
 K_R:=(Z\times Z)\setminus (A_R'\cup A_R'')
 = \{(y,z):\ \|\xi\|\leq R,\ \|\eta+\zeta\|\leq R\}
\end{equation}
is compact, this concludes the proof of uniform tightness. Weak transversality follows by Prokhorov's theorem.
\end{proof}
Strong transversality requires regularity of $\Phi$ in the direction orthogonal to $E$. To simplify notation, we introduce the function$\Psi:E\times E_0^\perp\to\R$,
\begin{equation}
 \Psi(\xi,\eta):=\Phi(\xi+\eta) 
\end{equation}
and denote by $D_2\Psi(\xi,\eta)$ the partial derivative of $\Psi(\xi,\eta)$ with respect to $\eta$.

\begin{prop}[Diagonal concentration of sub-Gaussian measures]\label{VkH97k}
Suppose that the assumptions of Prop.~\ref{p5qmD1} hold. In addition, assume:
\begin{itemize}
\item[i)] $\Psi(\xi,\eta)$ is {continuously} differentiable with respect to $\eta$.
\item[ii)] There {are constants} $0<C<\infty$, $\gamma \ge 0$ and a function $u:E\to[0,\infty)$ satisfying
$\int_E u(\xi)e^{-c\|\xi\|^2/4}\, d\mathcal{H}^{N}(\xi) <\infty$,
where $c$ is as in \eqref{9z2bLs},
such that
\begin{equation}\label{ceBj5Z}
    \| D_2\Psi(\xi,\eta) \|\leq {C \left(1+u(\xi)+\| \eta \|^\gamma\right)}, \quad \forall \xi \in E, \eta\in E_0^\perp .
\end{equation}
\end{itemize}
Then, $ \mu$ is strongly transversal and its diagonal concentration $\mu_\infty$ is given by
\begin{equation}\label{8ddOkTdm}
    \int_{Z\times Z} f(y,z) \, d\mu_\infty(y,z)
    = 
    \int_E
        f(\xi,\xi)
        \, {\rm e}^{-\Phi(\xi)}
    \, d\mathcal{H}^{N}(\xi) ,
\end{equation}
for every $f \in C_b(Z\times Z)$. Further, there is $C_1>0$ such that, for $\beta$ sufficiently large,
\begin{equation}\label{Kgx9hb}
    \| \mu_\beta - \mu_\infty \|_{\rm FN}
{\le}
    C_1 \beta^{-\frac12}\,.
\end{equation}
\end{prop}

\begin{proof}
It suffices to prove \eqref{Kgx9hb} for $\mu_\infty$ defined by \eqref{8ddOkTdm}. We fix
$f \in BL(Z\times Z)$ with $\|f\|_\infty \leq 1$, ${\rm Lip}(f) \leq {1/\ell}$ and
 rewrite the definition of $\mu_\infty$ as
\begin{equation}\label{8ddOkT}
    \int_{Z\times Z} f(y,z) \, d\mu_\infty(y,z)
    = 
    \int_{E\times E_0^\perp\times E_0}
 B_\beta^{-1}
 {\rm e}^{-\beta(\|\eta\|^2+\|\zeta\|^2) }
    f(\xi,\xi)
        \, {\rm e}^{-\Phi(\xi)}
    \, d\calH^{3N}(\xi,\eta,\zeta).
\end{equation}
We then
compute, as in the first step of
\eqref{eqmubsubgchv},
\begin{equation}
\begin{split}
    A_f^\beta
    &:=    \int_{Z\times Z} f(y,z) \, d(\mu_\beta-\mu_\infty)(y,z)
     \\ &=
    B_\beta^{-1}
    \int_{E\times E_0^\perp\times E_0}
    \left[ g(\xi,\eta,\zeta){\rm e}^{
    -\Psi(\xi,\eta) }
    -g(\xi,0,0)
    {\rm e}^{
    -\Psi(\xi,0) }
    \right]
    {\rm e}^{-\beta(\|\eta\|^2+\|\zeta\|^2)}
    d\calH^{3N}(\xi,\eta,\zeta) ,
\end{split}
\end{equation}
where we write, for $(\xi,\eta,\zeta)\in E\times E_0^\perp\times E_0$,
    $g(\xi,\eta,\zeta) := f(\xi+\eta, \xi-\zeta)$.
As $f$ is $1/\ell$-Lipschitz, so is $g(\xi,\cdot,\cdot)$, and 
\begin{equation}
 |g(\xi,\eta,\zeta)-g(\xi,0,0)|\le 
 \frac{\|\eta\|+\|\zeta\|}{\ell}.
\end{equation}
As $\Psi$ is differentiable in the second argument,
using ii)
\begin{equation}\begin{split}
\left|{\rm e}^{
 -\Psi(\xi,\eta) }
 -{\rm e}^{
 -\Psi(\xi,0) }\right|
 \le &\|\eta\| \int_0^1 
 {\rm e}^{
 -\Psi(\xi,t\eta) } \|D_2\Psi(\xi, t\eta)\| dt \\
 \le &C \|\eta\|(1 + u(\xi) +\|\eta\|^\gamma) \max_{t\in[0,1]}
 {\rm e}^{
 -\Psi(\xi,t\eta) }  .
 \end{split}
\end{equation}
From \eqref{9z2bLsn}, we further obtain
\begin{equation}
  -\Psi(\xi,t\eta) 
 \le b+\beta_0\|\eta\|^2+\beta_0\|\zeta\|^2-\frac  c4\|\xi\|^2 ,
 \quad\text{ for all } t\in[0,1] ,
\end{equation}
and with a triangular inequality
and $\|g\|_\infty\le 1$, we have
\begin{equation}\begin{split}
 &\left|g(\xi,\eta,\zeta){\rm e}^{
 -\Psi(\xi,\eta) }
 -g(\xi,0,0)
 {\rm e}^{
 -\Psi(\xi,0) }
 \right|\\
 & \le
 \left|g(\xi,\eta,\zeta)
 -g(\xi,0,0)
 \right|{\rm e}^{
 -\Psi(\xi,\eta) }
 + 
 \left|g(\xi,0,0)\right|\,\left|{\rm e}^{
 -\Psi(\xi,\eta) }
 -
 {\rm e}^{
 -\Psi(\xi,0) }
 \right|
\\
&\le
 \left(\frac{\|\eta\|+\|\zeta\|}{\ell} 
 +C\|\eta\|+C\|\eta\|^{\gamma+1}+C\|\eta\|u(\xi)\right)
 {\rm e}^{b+\beta_0\|\eta\|^2+\beta_0\|\zeta\|^2-\frac c4\|\xi\|^2 }.
 \end{split}\end{equation}
Therefore, for $\beta\ge 2\beta_0$,
\begin{equation}
 \begin{split}
  |A_f^\beta|\le 
 B_\beta^{-1}
 \int_{E\times E_0^\perp\times E_0}
&\left(\frac{\|\eta\|+\|\zeta\|}{\ell} 
 +C\|\eta\|+C\|\eta\|^{\gamma+1}+C\|\eta\|u(\xi)\right) \times \\
 &\qquad
 {\rm e}^{b-\frac c4\|\xi\|^2 }
 {\rm e}^{-\frac12\beta(\|\eta\|^2+\|\zeta\|^2)}
 d\calH^{3N}(\xi,\eta,\zeta) .
 \end{split}
\end{equation}
By the assumption on $u(\cdot)$,
the integral in $\xi$ gives a fixed constant. In the others we change variables, according to $\hat\eta:=\beta^{1/2}\eta$, $\hat\zeta:=\beta^{1/2}\zeta$, and  we obtain
\begin{equation}
 \begin{split}
  |A_f^\beta|\le 
C \beta^{-1/2}\,,
 \end{split}
\end{equation}
which concludes the proof.
\end{proof}

\begin{remark}\label{rmk:rate}
The rate $\beta^{-1/2}$ is a result of Taylor expanding $e^{-\Psi}$ in the second argument, and can therefore only be improved using our approach if $\mu_D$ is such that the corresponding $\Psi$ satisfies $D_2\Psi(\xi,0)=0$ for all $\xi\in E$.
\end{remark}

\section{Approximation}\label{kA7NJU}

We proceed to exploit the thermalization framework set forth in the foregoing for purposes of approximation by discrete measures.

\subsection{General strategy}
We first observe that transversality and approximation can be decoupled. Indeed, assume that $\mu \in \mathcal{M}^+(Z\times Z)$ is transversal, and let $\beta_h\to+\infty$, $\mu_\infty \in \mathcal{M}_b^+(Z\times Z)$, be such that $\mu_{\beta_h} \rightharpoonup \mu_\infty$. Suppose, in addition, that for some approximating sequence $(\mu_h)\in \mathcal{M}^+(Z\times Z)$ we have
\begin{equation}\label{IS5iwO}
    \mu_{h,\beta_h}-\mu_{\beta_h} \rightharpoonup 0 .
\end{equation}
Writing $\mu_{h,\beta_h} - \mu_\infty = (\mu_{h,\beta_h} - \mu_{\beta_h}) + (\mu_{\beta_h} - \mu_\infty)$,
necessarily $\mu_{h,\beta_h} \rightharpoonup \mu_\infty$. Therefore, in the following we focus on \eqref{IS5iwO}.

We remark that $\mu$ is not known in practice and \eqref{IS5iwO} cannot be verified directly. However, the sequence $\mu_{h,\beta_h}$ is known, and---by extrapolating values computed for a few values of $h$---it is possible to verify if it is uniformly bounded, tight, Cauchy in the flat norm and, consequently, converges to a limit. It bears emphasis that this scheme is entirely built upon the data $(\mu_h)$ and at no time knowledge of the underlying measure $\mu$, which is unknown in general, is required. It should also be noted that the convergence of $\mu_{h,\beta_h}$ to the thermalization of $\mu$ depends critically on the choice of the annealing sequence $\beta_h$. It is, therefore, important to determine conditions on the annealing sequence ensuring convergence under appropriate growth and regularity assumptions.

\subsection{Approximation by discrete material measures}\label{h9Mes5}

We proceed to further assess the feasibility of the abstract frameworks set forth above by considering the specific case of approximation by discrete material measures. Measures supported on finite point sets are of practical significance, since they represent material data sets resulting from discrete experimental measurements. They also provide a natural means of approximating general measures in the sense of weak convergence.
Specifically, we consider a sequence of 
material likelihood measures 
$\mu_{D,h}$ of the form given in \eqref{muDh}.
We begin by showing that, if $(\mu_{\beta_h})$ is uniformly bounded and tight and $(\mu_{D,h})$ suitably approximates $\mu_D$, then $(\mu_{h,\beta_h})$ is also uniformly bounded and tight.

\begin{prop}[Boundedness and tightness]\label{D7NtRs}
Let $\mu = \mu_D \times \mu_E \in \mathcal{M}^+(Z\times Z)$ with $\mu_E$ as in (\ref{eqmuedetermintr}). Assume
\begin{enumerate}
\item\label{itbdtightub2}
$(\mu_\beta)$, for $\beta\ge\beta_0>0$, is uniformly bounded and uniformly tight.
\end{enumerate}
Let $\mu_h = \mu_{D,h} \times \mu_E \in \mathcal{M}^+(Z\times Z)$, with $\mu_{D,h}$ as in (\ref{muDh}). Let $\beta_h\to+\infty$ {and $\epsilon_h\to0$}. 
Assume further that, for every $h$, there is a partition $\mathcal{A}_h = \{A_\xip \, : \, \xip \in P_h \}$ of $Z$, with $\xip \in A_\xip$ and $\mu_D(A_\xip)<\infty$ for every $\xip\in P_h$ and every $h$, such that:
\begin{enumerate}\stepcounter{enumi}
\item\label{itbdtightdelta}
There is $\bar c>0$  such that, for all $h$ and all $\xip \in P_h$,
\begin{equation}
    m_\xip\le \bar c\mu_D(A_\xip) .
\end{equation}
\item\label{itbdtighte2b}
The sequence $(\beta_h \epsilon_h^2)$ is bounded.
\item\label{itbdtightcst}
There is $c_*\ge1$ such that for all $h \in \mathbb{N}$, $\xip \in P_h$,  $z\in E$, all $y\in A_p$,
\begin{equation}\label{qc3qLRcst22}
    \|\xip-z\|^2\le c_*^2\left[\epsilon_h^2+\|y-z\|^2\right].
\end{equation}
\item\label{itbdtightcst2}
For all $h \in \mathbb{N}$, $\xip \in P_h$, 
\begin{equation}\label{ass:vii}
    \int_{A_p} \| y - \xip \|^2 \, d\mu_D(y)
    \leq
    \epsilon_h^2 \, \mu_D(A_p) .
\end{equation}
\end{enumerate}
Then there is $h_0>0$ such that, after restricting to $h\ge h_0$,
\begin{itemize}
\item[a)] The sequence 
$(\mu_{h,\beta_h})$ is uniformly bounded in $\mathcal{M}_b(Z\times Z)$.
\item[b)] The sequence $(\mu_{h,\beta_h})$ is uniformly tight in $\mathcal{M}_b(Z\times Z)$.
\end{itemize}
\end{prop}

\begin{proof}
a) For $\mu_E$ as in (\ref{eqmuedetermintr}) and $\mu_{D,h}$ as in (\ref{muDh}), for any $\beta>0$ we have
\begin{align}\label{eqmuhbetah}
\notag
    \| \mu_{h,\beta} \|_{\rm TV}
    &=
\sum_{\xip\in P_h}
     \mu_{h,\beta}(A_p\times E)
     =\sum_{\xip\in P_h}
   \mu_{h,\beta}(\{p\}\times E)\\
   &=\sum_{\xip\in P_h}
    \int_E
        m_\xip
        B_{\beta}^{-1}
        {\rm e}^{-\beta \|\xip - z\|^2}
    \, d\mathcal{H}^N(z)\,.
\end{align}
We denote by $P_E$ the closest-point projection from $Z$ onto $E$.  
By orthogonality, $\|\xip - z\|^2 = \|\xip - P_E(\xip)\|^2 + \|P_E(\xip) - z\|^2$, whence
\begin{equation}\label{LI3a6H}
\mu_{h,\beta}(A_p\times E)=
        m_\xip B_{\beta}^{-1} C_{\beta} 
        {\rm e}^{-\beta \|\xip - P_E(\xip)\|^2} ,
\quad    C_{\beta} := \int_{E_0} {\rm e}^{-\beta \|z\|^2} d\calH^N(z).
\end{equation}
We claim that for all $\xip \in P_h$,
\begin{equation}\label{eqsecmoment}
\frac{1}{2\mu_D(A_\xip)}
    \int_{A_\xip}
            \|y - P_E(y)\|^2
        \, d\mu_D(y)
\le 
        \|\xip - P_E(\xip)\|^2 
    +
    \epsilon_h^2 .
\end{equation}
To see this, we use that 
$\Id-P_E=P_{E^\perp}$ is a projection
and \ref{itbdtightcst2} to get
\begin{equation}
\begin{split}
    \int_{A_\xip}\|P_{E^\perp}(y-p)\|^2 \, d\mu_D(y)
    \le
    \int_{A_\xip}\|y-\xip\|^2  \, d\mu_D(y) \le\epsilon_h^2   \mu_D(A_\xip)\,,
\end{split}
\end{equation}
and 
with 
$\frac12\|y - P_E(y)\|^2=
\frac12\|P_{E^\perp}(y)\|^2\le
    \|P_{E^\perp}(y-\xip)\|^2 +
    \|P_{E^\perp}(\xip)\|^2$, equation
\eqref{eqsecmoment} follows.

An application of Jensen's inequality to the convex function $-\mathop{\log}t$ and \eqref{eqsecmoment}
 give, for any fixed $\xip\in P_h$,
\begin{equation}\label{3VyqRFa}
\begin{split}
    & -
    \log
    \Big(
        \frac{1}{\mu_D(A_\xip)}
        \int_{A_\xip}
            {\rm e}^{-\frac{\beta}{2} \|y - P_E(y)\|^2}
        \, d\mu_D(y)
    \Big)
    \\ & \leq
    \frac{\beta }{2\mu_D(A_\xip)}
    \int_{A_\xip}
    \|y - P_E(y)\|^2
    \, d\mu_D(y)
    \leq
    \beta
    \Big(
        \|\xip - P_E(\xip)\|^2 + \epsilon_h^2
    \Big) ,
\end{split}
\end{equation}
which, rearranging terms, reads
\begin{equation}\label{3VyqRF}
\begin{split}
    {\rm e}^{-\beta
        \|\xip - P_E(\xip)\|^2     }
        \le
    {\rm e}^{\beta \epsilon_h^2}
        \frac1{\mu_D(A_\xip) }
        \int_{A_\xip}
            {\rm e}^{-\frac{\beta}{2} \|y - P_E(y)\|^2}
        \, d\mu_D(y).
\end{split}
\end{equation}
Using this bound, \ref{itbdtightdelta}, the definition of $C_\beta$, and the fact that  $B_{\beta/2}=2^{N}B_{\beta}$, (\ref{LI3a6H}) becomes
\begin{align}\label{eqmuhbape}
  \mu_{h,\beta}(A_p\times E)
    &\leq
    {\rm e}^{\beta \epsilon_h^2}
    \frac{m_\xip      B_{\beta}^{-1} C_{\beta}}{\mu_D(A_\xip)  }
    \int_{A_\xip}
        {\rm e}^{-\frac{\beta}{2} \|y - P_E(y)\|^2}
    \, d\mu_D(y) \notag\\
    &\leq \bar c 
    {\rm e}^{\beta \epsilon_h^2}
    \int_{A_p}
        B_{\beta}^{-1} \int_E
        {\rm e}^{-\frac{\beta}{2}\|y - P_E(y)\|^2-\beta \|P_E(y)-z\|^2}
    \, d\calH^N(z) d\mu_D(y)\notag\\
     &\leq  \bar c 
    {\rm e}^{\beta \epsilon_h^2}2^{N}
    \int_{A_p}
        B_{\beta/2}^{-1} \int_E
        {\rm e}^{-\frac{\beta}{2} \|y - z\|^2}
    \, d\calH^N(z) d\mu_D(y)
    \notag \\ & = 
     \bar c 
     {\rm e}^{\beta \epsilon_h^2}2^{N} \mu_{\beta/2}(A_p\times E)\,.
\end{align}
Summing over all $p\in P_h$ finally gives
\begin{equation}\label{eqbdmuhbetahdf}
    \| \mu_{h,\beta} \|_{\rm TV}
    \leq
    \bar c
    {\rm e}^{\beta \epsilon_h^2} 2^{N}\,
    \| \mu_{\beta/2} \|_{\rm TV}.
\end{equation}
Using this for $\beta=\beta_h\ge 2\beta_0$
proves the claim since, by \ref{itbdtightub2}, $\| \mu_{\beta} \|_{\rm TV}$ for $\beta\ge\beta_0$ is uniformly bounded and by \ref{itbdtighte2b}
 $(\beta_h \epsilon_h^2)$ is also uniformly bounded.

b) By the uniform tightness of $(\mu_\beta)$, for every $\delta > 0$ there are compact sets $K\subset Z$ and $L\subset E$ such that
\begin{equation}\label{eqmuzzkld}
    \mu_{\beta}(Z\times Z \setminus K \times L ) \leq \delta ,
\end{equation}
for all $\beta\ge\beta_0$. The strategy of the proof is as follows: we show that for all $\epsilon>0$, there exists a choice of small enough $\delta>0$ and corresponding sets $K'\subset Z$ and $L'\subset E$ depending on $\delta$ such that
for sufficiently large $h$  one has
\begin{equation}\label{approxtight}
\begin{split}
    \mu_{h,\beta_h}(Z\times Z \setminus K' \times L')
    \leq \epsilon .
    \end{split}
\end{equation}
More precisely, for 
 $c_*$ as in assumption \ref{itbdtightcst} and
some $C_*>0$ chosen below, we define
\begin{equation}
\begin{split}
  &K':=\big\{y\in Z: \|y\| \le  C_*+ c_*\max \{\|z\|: z\in K\}\big\}\,,\\  
  &L':=\{z\in E: \dist(z,L\cup K')\le C_*\}\,.
\end{split}
\end{equation}
Further, defining $P_h^K:=\{\xip\in P_h: A_p\cap K\ne\emptyset\}$, let $K_h :=\cup_{\xip\in P_h^K}A_\xip$ be the span of the $\mathcal{A}_h$-cell cover of $K$. It follows that $K\subseteq K_h$.
We pick a $z_0\in E$ and assume that $h_0$ is chosen so that $\epsilon_h\le 1$ for all $h\ge h_0$.
Then for any $y\in A_\xip$ by \ref{itbdtightcst}
\begin{equation}\label{eqxipz0}
\begin{split}
 \|\xip\|&\le \|z_0\|+\|\xip-z_0\|
 \le\|z_0\|+c_*\sqrt{\|y-z_0\|^2+1} \\
 &\le \|z_0\|+c_*\|y-z_0\|+c_*
 \le (1+c_*)\|z_0\|+c_*+c_*\|y\|=C_*+c_*\|y\|,\end{split}
\end{equation}
where we chose $C_*:=c_*+(1+c_*)\|z_0\|$.
Let $p\in P^K_h$. Then there is $y\in A_p\cap K$, and by \eqref{eqxipz0} $\|p\|\le C_*+c_*\|y\|$. This implies $p\in K'$, and therefore
 $P^K_h\subseteq K'$.

 The proof for bound \eqref{approxtight} consists of two parts. First we localize in the first argument, and show that there exists a constant $C>0$ only depending on $N$ and $\bar c$ such that for all $h$ sufficiently large,
\begin{equation}\label{bStep1}
\begin{split}
      \mu_{h,\beta_h}((Z\setminus  P^K_h)\times E )
    &\leq C    \delta \,.
\end{split}
\end{equation}
Then we localize in the second argument, and show that there exists a constant $C>0$ only depending on $N$, $\bar c$ and $C_*$ such that for $h$ sufficiently large,
\begin{equation}\label{bStep2}
\begin{split}
    &
    \mu_{h,\beta_h} (P^K_h\times (Z \setminus L'))
    \le C \delta ,
\end{split}
\end{equation}
The two estimates are then combined in \eqref{eqdecsetdiff}-\eqref{eqmudecsetdiff} below.  

To show \eqref{bStep1}, we proceed as in \eqref{eqmuhbetah} with the sum restricted to $p\in P_h\setminus P_h^K$ and, together with the same estimate as in \eqref{eqmuhbape}, we obtain
\begin{align}
    \mu_{h,\beta_h}((Z\setminus P^K_h)\times E )
    &=
    \sum_{\xip\in P_h\setminus P_h^K} \mu_{h,\beta_h}(A_\xip\times E)
    \leq 
    \bar c \,
    {\rm e}^{\beta_h \epsilon_h^2}2^{N} 
     \sum_{\xip\in P_h\setminus P_h^K}  \mu_{\beta_h/2}(A_\xip\times E)\notag\\
&    = 
    \bar c \,
    {\rm e}^{\beta_h \epsilon_h^2}2^{N} 
     \mu_{\beta_h/2}((Z\setminus K_h)\times E)\,.
\end{align}
Recalling that  $K\subseteq K_h$,
for $h$ sufficiently large that $\beta_h\ge2\beta_0$ one has
\begin{equation}
\begin{split}
      \mu_{\beta_h/2}((Z\setminus K_h)\times E )
    &\leq 
    \mu_{\beta_h/2}((Z\setminus K)\times Z )
    \le 
    \mu_{\beta_h/2}(Z\times Z \setminus K\times L )
    \leq 
    \delta .
\end{split}
\end{equation}
Together with assumption \ref{itbdtighte2b}, we conclude that \eqref{bStep1} holds.

To prove \eqref{bStep2}, we 
compute as in \eqref{eqmuhbetah}, 
using  \ref{itbdtightdelta},
\begin{equation}\label{eqmuhbetahb}
\begin{split}
\mu_{h,\beta_h} (P^K_h\times (Z \setminus L'))
    =&
    \sum_{\xip\in P_h^K}
    \int_{E\setminus L'}
        m_\xip
        B_{\beta_h}^{-1}
        {\rm e}^{-\beta_h \|\xip - z\|^2}
    \, d\mathcal{H}^N(z)\\
    \le& 
    \sum_{\xip\in P_h^K}
    \int_{E\setminus L'}
        \bar c \mu_D(A_p)
        B_{\beta_h}^{-1}
        {\rm e}^{-\beta_h \|\xip - z\|^2}
    \, d\mathcal{H}^N(z)\,.
    \end{split}
\end{equation}
The conditions $p\in P^K_h\subseteq K'$, $z\in E\setminus L'$ and the definition of $L'$ imply
\begin{equation}\label{eqpepzcst}
 \|p-z\|> C_*.
\end{equation} 
We next show that 
there is $a\in(0,1)$ such that
for any $z\in E$, any $h$, any $\xip\in P_h$, 
if \ref{itbdtightcst2} and \eqref{eqpepzcst}  hold then
\begin{equation}\label{eqcellbrsdf}
{\rm e}^{- \beta_h\|p-z\|^2} \le  G:=
\frac1{ \mu_D(A_\xip)}
\int_{A_\xip} {\rm e}^{-a \beta_h\|y-z\|^2} d\mu_D(y) .
\end{equation}
To see this, we 
estimate, using
again
convexity of $t\mapsto -\log t$,
\begin{equation}
 -\log G
 \le \frac{1}{\mu_D(A_\xip)}
 \int_{A_\xip} a\beta_h\|y-z\|^2 d\mu_D(y) .
\end{equation}
With $\|y-z\|^2\le 2 \|y-p\|^2+2\|p-z\|^2$
and 
\ref{itbdtightcst2} this becomes
\begin{equation}
 -\log G
 \le 
 2a\beta_h \epsilon_h^2+2a\beta_h\|p-z\|^2
,
\end{equation} 
and with \eqref{eqpepzcst} and 
$\epsilon_h\le 1$
we obtain
\begin{equation}
 -\log G
 \le 2a\left(1+C_*^{-2}\right)\beta_h\|p-z\|^2\,.
\end{equation}
We finally choose $a:=
(1+C_*^{-2})^{-1}/2
$ and taking an exponential on both sides concludes the proof of \eqref{eqcellbrsdf}.
Using \eqref{eqcellbrsdf} and \eqref{eqmuhbetahb} gives for $a\beta_h\ge\beta_0$,
\begin{equation}
\begin{split}
    &
    \mu_{h,\beta_h} (P^K_h\times (Z \setminus L'))
    \le 
    \bar c    
    \sum_{\xip\in P_h^K}
    \int_{E\setminus L'}              
        B_{\beta_h}^{-1}
    \int_{A_\xip} {\rm e}^{-a \beta_h\|y-z\|^2} \, d\mu_D(y) 
    \, d\mathcal{H}^N(z)
     \\ & \le
    \bar c    
    \int_{E\setminus L'}              
        B_{\beta_h}^{-1}
    \int_{Z} {\rm e}^{-a \beta_h\|y-z\|^2} \, d\mu_D(y) 
    \, d\mathcal{H}^N(z)
    = 
    \bar c a^{-N} \mu_{a\beta_h}(Z\times (E\setminus L'))
    \\ & \le
    \bar c a^{-N} \mu_{a\beta_h}(Z\times (E\setminus L))
    \le 
   \bar c a^{-N} \mu_{a\beta_h}(Z\times Z \setminus K\times L)
    \le 
    \bar c a^{-N} \delta ,
\end{split}
\end{equation}
where we used that $L\subseteq L'$ and \eqref{eqmuzzkld}. This proves \eqref{bStep2}.
 
 Combining estimates \eqref{bStep1} and \eqref{bStep2}, using that
 \begin{equation}\label{eqdecsetdiff}
  Z\times Z\setminus K'\times L'
  \subseteq 
  Z\times Z\setminus P^K_h\times L'
  =((Z\setminus P^K_h)\times Z )\bigcup 
  (P^K_h\times (Z\setminus L'))\,,
 \end{equation}
and that all measures vanish on $Z\times (Z\setminus E)$, we conclude
\begin{equation}\label{eqmudecsetdiff}
\begin{split}
    \mu_{h,\beta_h}(Z\times Z \setminus K' \times L')
    \leq&
    \mu_{h,\beta_h}((Z\setminus P^K_h)\times Z )
    +\mu_{h,\beta_h}(P^K_h\times (Z \setminus  L'))
    \leq 
    C \delta,
    \end{split}
\end{equation}
for some (different) constant $C>0$,
provided $h$ is sufficiently large. Therefore for every $\epsilon > 0$, we can choose $\delta$ small enough such that this is smaller than $\epsilon$ for all $h \geq h_0$, as required.
\end{proof}

We note that conditions \ref{itbdtightdelta}, \ref{itbdtightcst} and \ref{itbdtightcst2} stipulate precisely how the sequence $(P_h)$ of point-data sets must approximate the material likelihood measure $\mu_D$ in order to ensure convergence. Condition \ref{itbdtighte2b}   relates the sequences $(\beta_h)$ and $(\epsilon_h)$ and, specifically, requires $\beta_h$ to diverge no faster than $\epsilon_h^{-2}$. 

Finally, we turn to quantitative convergence in the flat norm. 
\begin{cor}[Convergence in flat norm]\label{y1TP50}
Suppose that the assumptions of Prop.~\ref{D7NtRs} hold. Assume in addition that there is $\delta_h\to 0$ such that, for all $\xip \in P_h$,
\begin{equation}\label{condvi}
    |m_\xip-\mu_D(A_\xip)|
    \leq
    \delta_h \, \mu_D(A_\xip) .
\end{equation}
Then for sufficiently large $h$, we have
\begin{equation}\label{th4Iw8}
    \| \mu_{h,\beta_h} - \mu_{\beta_h} \|_{\rm FN}
    \leq C
(\delta_h + \beta_h^{1/2}\epsilon_h)
\sup_{\beta\ge\beta_0} \|\mu_\beta\|_{\rm TV}.
\end{equation}
\end{cor}
\begin{remark}\label{remassprop} 
Assumption 
\eqref{condvi} obviously implies assumption \ref{itbdtightdelta}.
\end{remark}
\begin{proof}[Proof of Cor.~\ref{y1TP50}]
For any $f\in BL(Z\times Z)$ with $\|f\|_\infty\le 1$ and ${\rm Lip} f\le1/\ell$,
\begin{equation}
\begin{split}
    &
    \int_{Z\times Z} f(y,z) d(\mu_{h,\beta_h} - \mu_{\beta_h})(y,z)
    \\ &=
    \int_E
        \sum_{\xip\in P_h}
        \Big(
            m_\xip f(\xip,z)
            B_{\beta_h}^{-1}
            {\rm e}^{-\beta_h \|\xip - z\|^2}
        - 
        \int_{A_\xip}
            f(y,z)
            B_{\beta_h}^{-1}
            {\rm e}^{-\beta_h \|y - z\|^2}
        \, d\mu_D(y)
    \Big)
    \, d\mathcal{H}^N(z) .
\end{split}
\end{equation}
We write, recalling $\|f\|_\infty\le 1$,
\begin{equation}\label{oS852h}
\begin{split}
    &
    \Big| \int_{Z\times Z} f(y,z) d(\mu_{h,\beta_h} - \mu_{\beta_h})(y,z) \Big|
    \leq  A^h + B^h,
\end{split}
\end{equation}
with
\begin{equation}\label{oS852hn}
\begin{split}
    A^h:=  &
    \int_E
        \sum_{\xip\in P_h}
            | m_\xip - \mu_D(A_\xip)|
            B_{\beta_h}^{-1}
            {\rm e}^{-\beta_h \|\xip - z\|^2}
    \, d\mathcal{H}^N(z),
     \\ B^h:=&
    \int_E
        \sum_{\xip\in P_h}
        \int_{A_\xip}
            B_{\beta_h}^{-1}
            \Big|
                f(\xip,z)
                {\rm e}^{-\beta_h \|\xip - z\|^2}
                - 
                f(y,z)
                {\rm e}^{-\beta_h \|y - z\|^2}
            \Big|
        \, d\mu_D(y)
    \, d\mathcal{H}^N(z) .
\end{split}
\end{equation}
The first term is estimated  using \eqref{condvi},
\begin{equation}\label{eqbdah}
\begin{split}
    &
    A^h \leq 
    \frac{\delta_h}{1-\delta_h}
    \int_E
        \sum_{\xip\in P_h}
            m_\xip
            B_{\beta_h}^{-1}
            {\rm e}^{-\beta_h \|\xip - z\|^2}
    \, d\mathcal{H}^N(z)
    =
    \frac{\delta_h}{1-\delta_h} \| \mu_{h,\beta_h} \|_{\rm TV} .
\end{split}
\end{equation}
Next, we proceed to estimate the second term in the bound. Writing
\begin{equation}
\begin{split}
    &
    f(\xip,z)
    {\rm e}^{-\beta_h \|\xip - z\|^2}
    -
    f(y,z)
    {\rm e}^{-\beta_h \|y - z\|^2}
      \\ &=
    ( f(\xip,z) - f(y,z) )
    {\rm e}^{-\beta_h \|\xip - z\|^2}
    +
    f(y,z)
    \Big(
        {\rm e}^{-\beta_h \|\xip - z\|^2}
        -
        {\rm e}^{-\beta_h \|y - z\|^2}
    \Big)\, ,
\end{split}
\end{equation}
and recalling ${\rm Lip} f\le1/\ell$, $\|f\|_\infty\le 1$,  we obtain
\begin{equation}\label{mV9mLr}
\begin{split}
    B^h
    & \le
    \int_E
        \sum_{\xip\in P_h}
        \int_{A_\xip}
            B_{\beta_h}^{-1}
            \frac{1}{\ell}\| \xip - y \|
            {\rm e}^{-\beta_h \|\xip - z\|^2}
        \, d\mu_D(y)
    \, d\mathcal{H}^N(z)
    \\ & +
    \int_E
        \sum_{\xip\in P_h}
        \int_{A_\xip}
            B_{\beta_h}^{-1}
            \Big|
                {\rm e}^{-\beta_h \|\xip - z\|^2}
                -
                {\rm e}^{-\beta_h \|y - z\|^2}
            \Big|
        \, d\mu_D(y)
    \, d\mathcal{H}^N(z) .
\end{split}
\end{equation}
In order to estimate the last term, we compute the remainder term in Taylor's expansion as
\begin{equation}\label{eqzxiespfd}
\begin{split}
            {\rm e}^{-\beta_h \|\xip - z\|^2}
            -
            {\rm e}^{-\beta_h \|y - z\|^2}
=\beta_h (\|y - z\|^2-\|\xip - z\|^2)\int_0^1 
            {\rm e}^{-\beta_h (t\|\xip - z\|^2+(1-t) \|y - z\|^2)} dt
\end{split}
\end{equation}
and since by Assumption~\ref{itbdtightcst} we have
$-\|y-z\|^2
\le \epsilon_h^2-\|p-z\|^2/c_*^2$
for all $y\in A_\xip$, and $c_*\ge1$,
\begin{equation}
\int_0^1{\rm e}^{-\beta_h (t\|\xip - z\|^2+(1-t) \|y - z\|^2)}dt  \le 
D(p,z):={\rm e}^{\beta_h \epsilon_h^2}
{\rm e}^{-\beta_h \|\xip - z\|^2/c_*^2}.
\end{equation}
Factorizing the first term in \eqref{eqzxiespfd} and then using a triangular inequality,
\begin{equation}\label{eqzxiespfd3}
\begin{split}
        \Big|
            {\rm e}^{-\beta_h \|\xip - z\|^2}
            -
            {\rm e}^{-\beta_h \|y - z\|^2}
        \Big|
&\le D(p,z) 
\beta_h \|y-p\| \, (\|y-z\|+\|p-z\|)
\\
&\le D(p,z) 
(\beta_h \|y-p\|^2 + 2\beta_h  \|y-p\|\,\|p-z\|).
\end{split}
\end{equation}
Integrating \eqref{eqzxiespfd3} and using Hölder and Assumption \ref{itbdtightcst2},
\begin{equation}\label{eqzxiespfd4}
\begin{split}
    &
     \int_{A_\xip}
        \Big|
            {\rm e}^{-\beta_h \|\xip - z\|^2}
            -
            {\rm e}^{-\beta_h \|y - z\|^2}
        \Big|
     \, d\mu_D(y)
\le 2D(p,z) \beta_h\epsilon_h(\epsilon_h+\|p-z\|  )\mu_D(A_p).
\end{split}
\end{equation}
In order to incorporate the term $\|p-z\|$ into the exponential, we remark that   $\log t\le t-1\le t^2/2$ implies
$t{\rm e}^{-t^2}\le {\rm e}^{-t^2/2}$ for all $t>0$. Using this with $t=\beta_h^{1/2}\|\xip-z\|/c_*$ yields
\begin{equation}\begin{split}
      {\rm e}^{-\beta_h \|\xip - z\|^2/c_*^2}   
        \beta_h^{1/2} \|p-z\|
        \le c_*
  {\rm e}^{-\beta_h \|\xip - z\|^2/(2c_*^2)}.
 \end{split}
\end{equation}
For the other term in \eqref{eqzxiespfd4} we use  $\beta_h^{1/2}\epsilon_h\le C^{1/2}$. We conclude, 
\begin{equation}\label{eqbhsectif}
\begin{split}
    &
    \int_{A_\xip}
        \Big|
            {\rm e}^{-\beta_h \|\xip - z\|^2}
            -
            {\rm e}^{-\beta_h \|y - z\|^2}
        \Big|
    \, d\mu_D(y)
    \leq 
    c'' \beta_h^{1/2}\epsilon_h        {\rm e}^{-\beta_h \|\xip - z\|^2/(2c_*^2)}
    \mu_D(A_\xip) ,
\end{split}
\end{equation}
for all $z$ and all $\xip\in P_h$, with $c''\ge1$ a constant depending on $C$ and $c_*$. 
We turn back to
(\ref{mV9mLr}), use Hölder and 
Assumption \ref{itbdtightcst2} in the first term, and \eqref{eqbhsectif} in the second one, as well as
$\mu_D(A_\xip)\le m_\xip/(1-\delta_h)$. We obtain
\begin{equation}\label{eqbhfin}
    B^h
    \le
    c''    
    \int_E
        \sum_{\xip\in P_h}\frac{m_\xip}{1-\delta_h}
            B_{\beta_h}^{-1}
            \Big(
                \frac{{\epsilon_h}}{\ell}
                +
                \beta_h^{1/2} \epsilon_h
            \Big)
            {\rm e}^{-\frac{\beta_h}{2c_*^2} \|\xip - z\|^2}
    \, d\mathcal{H}^N(z) .
\end{equation}
Combining \eqref{oS852h}, \eqref{eqbdah}, \eqref{eqbhfin} and \eqref{eqmuhbetah},
\begin{equation}\begin{split}
    \Big|
        \int_{Z\times Z} f d(\mu_{h,\beta_h} - \mu_{\beta_h})
    \Big|
    \leq &
    \frac{\delta_h}{1-\delta_h} \| \mu_{h,\beta_h} \|_{\rm TV}
    + 
        \frac{c'''}{1-\delta_h}
            \Big(
                \frac{\epsilon_h}{\ell}
                +
                \beta_h^{1/2} \epsilon_h
            \Big)
\| \mu_{h,\beta_h/(2c_*^2)} \|_{\rm TV}.
 \end{split}           \end{equation}
Using \eqref{eqbdmuhbetahdf} with $\beta=\beta_h$ and with $\beta=\beta_h/(2c_*^2)$ and then \ref{itbdtighte2b} and $\delta_h\le\frac12$ gives
\begin{equation}\begin{split}
    \Big|
        \int_{Z\times Z} f d(\mu_{h,\beta_h} - \mu_{\beta_h})
    \Big|
    & \leq 
    \tilde c {\delta_h}
    \| \mu_{\beta_h/2}  \|_{\rm TV}
+
        \tilde c 
            \Big(
                \frac{\epsilon_h}{\ell}
                +
                \beta_h^{1/2} \epsilon_h
            \Big)
\| \mu_{\beta_h/(4c_*^2)}  \|_{\rm TV}
\end{split}
\end{equation}
for some $\tilde c>0$.
We conclude using \ref{itbdtightub2}
and recalling that $\beta_h\to\infty$.
\end{proof}

We note that, under the conditions of Prop.~\ref{D7NtRs} and Cor.~\ref{y1TP50}, convergence in the flat norm is attained if: $\delta_h \to 0$, which ensures that the discrete measures in $\mu_{D,h}$ carry the right amount of mass in the limit; $\epsilon_h\to 0$, corresponding to an increasingly fine discretization of $\mu_D$; and $\beta_h\epsilon_h^2 \to 0$, which requires $\beta_h$ to diverge to $+\infty$ more slowly than $\epsilon_h^{-2}$. This estimate permits to render the decoupling $\mu_{h,\beta_h} - \mu_\infty = (\mu_{h,\beta_h} - \mu_{\beta_h}) + (\mu_{\beta_h} - \mu_\infty)$ quantitative. Thus, consider for instance the case in which $\mu_{D,h}$ approximates a measure $\mu_D=e^{-\Phi}\mathcal{L}^{2N}$, with $\Phi$ satisfying both \eqref{9z2bLs} and \eqref{ceBj5Z}. In addition, choose $m_p=\mu_D(A_p)$ for simplicity, yielding $\delta_h=0$. Then, the bounds (\ref{Kgx9hb}) and (\ref{th4Iw8}) lead (for large $h$) to the estimate \eqref{eqintrooptbeta}, whereupon a minimization of the bound gives $\beta_h \sim \epsilon_h^{-1}$ and $\| \mu_{h,\beta_h} - \mu_{\beta_h} \|_{\rm FN} \sim \epsilon_h^{1/2}$. These estimates simultaneously set forth the optimal annealing rate and the resulting convergence rate in the flat norm, as anticipated in Section~\ref{sec:intro}.

\section{Example: Transportation networks}\label{Waxd2V}

As an illustrative example of application, we consider a transportation network, e.~g., pipeline systems, traffic networks, electrical circuits and grids, consisting of $n$ free nodes and $N$ oriented edges. Every free node in the network carries a potential. In addition, the network may contain grounded nodes or nodes connected to a source, which are held at a fixed potential. We denote by $u \in \mathbb{R}^n$ the array of free-node potentials, by $\varepsilon \in \mathbb{R}^N$ the array of potential differences along the edges and by $\sigma \in \mathbb{R}^N$ the corresponding fluxes. 

\subsection{Field equations}
The nodal potentials, edge potential differences and edge currents satisfy the conservation and compatibility relations
\begin{equation}\label{Mnhw8F}
    B^T \sigma = f  \hskip5mm\text{ and } \hskip5mm
    \varepsilon = B u + g ,
\end{equation}
with sources $f\in \mathbb{R}^n$, applied potential differences $g \in \mathbb{R}^N$ and connectivity matrix $B \in \mathbb{R}^{N\times n}$. The corresponding phase space is $Z = \mathbb{R}^N \times \mathbb{R}^N$, metrized by the norm
\begin{equation}\label{zQpQFw}
    \| z \|
    :=
    \left(
    \sum_{e=1}^N
        \Big(
            \mathbb{C}_e |\varepsilon_e|^2
            +
            \mathbb{C}_e^{-1} |\sigma_e|^2
        \Big)
    \right)^{1/2} ,
\end{equation}
with coefficients $\mathbb{C}_e > 0$, $e=1,\dots,N$. 
Here and below, we write  $z=(\epsilon,\sigma)=((\varepsilon_1,\sigma_1),\dots,(\varepsilon_N,\sigma_N))$.
The conditions (\ref{Mnhw8F})
define the set of admissible states 
\begin{equation}\label{Rd5zWa}
    E := \{ z=(\varepsilon,\sigma)\in Z \, : \, \varepsilon = B u + g,\ B^T \sigma = f \} ,
\end{equation}
parametrized by $f \in \mathbb{R}^n$ and $g \in \mathbb{R}^N$. We assume that $f$ and $g$ are deterministic and, hence, the corresponding constraint measure $\mu_E$ is (\ref{eqmuedetermintr}).
We further assume that the matrices $B$ and  $\mathbb{C}:={\rm diag}\{\mathbb{C}_1, \dots, \mathbb{C}_N\}$ obey the non-degeneracy condition
\begin{equation}\label{eqnetcomp}
 B^T \mathbb{C} B > 0
\end{equation}
(i.~e., $B^T \mathbb{C} B$ is a positive definite matrix). This condition ensures existence and uniqueness of the classical solution, which is characterized by $\sigma=\mathbb C\epsilon$, $(\epsilon,\sigma)\in E$. Indeed, inserting in (\ref{Mnhw8F}) one obtains
$u=(B^T\mathbb C B)^{-1}(f-B^T\mathbb C g)$, 
with $\epsilon$ and $\sigma$ depending linearly on $u$.

\subsection{Material likelihood}
Suppose that the material behavior of every edge $e$ is characterized by a local material likelihood measure of the form
\begin{equation}\label{4JJeZV}
    \mu_{D,e}
    =
    {\rm e}^{-\Phi_{D,e}}
\, \mathcal{L}^2 =
    \exp
    \Big(
        -
        \frac{\mathbb{C}_e^{-1}}{2s_e^2}
        | \sigma_e-\mathbb{C}_e\varepsilon_e |^2
    \Big)
    \, \mathcal{L}^2 ,
\end{equation}
parameterized by $s_e > 0$. Thus, $\mu_{D,e}$ is a sliding Gaussian measure, i.~e., is Gaussian in the variable $\sigma_e - \mathbb{C}_e\varepsilon_e$ and invariant under translations along the line $\sigma_e = \mathbb{C}_e\varepsilon_e$. We note that $\mu_{D,e}$ is not finite, much less a probability measure. 
The local material likelihood measure (\ref{4JJeZV}) represents a stochastic Ohm's law for electrical circuits, a stochastic Darcy-Weisbach law for pipeline networks, and similar laws for other physical systems. The global material data measure is 
  $  \mu_D
    =
    \prod_{e=1}^N
    \mu_{D,e} $.

\subsection{Inference problem}

We verify that  the conditions for transversality stated in Prop.~\ref{p5qmD1} and Prop.~\ref{VkH97k}
are indeed satisfied. Thus, from (\ref{4JJeZV}) we compute the (global) material logarithmic potential as
\begin{equation}
    \Phi_D(y)=
        \sum_{e=1}^N \Phi_{D,e}(\varepsilon_e,\sigma_e)
    =
    \sum_{e=1}^N
        \frac{\mathbb{C}_e^{-1}}{2s_e^2}
         | \sigma_e-\mathbb{C}_e\varepsilon_e |^2 .
\end{equation}
In order to show that $\mu_D$ is sub-Gaussian, we first argue that 
 there are $c',b'>0$ such that
\begin{equation}\label{p38xEk}
    \| z \|^2
    \le 
    c'    
    \Phi_D(z)
    +
    b' ,
    \quad
    \forall z \in E .
\end{equation}
We start by showing that $\Phi_D$ is strictly positive on any nonzero element of $E_0$ (the linear space obtained by translation of $E$ to the origin). Assume $(\eps,\sigma)\in E_0$, which requires that $\eps=Bu$ for some $u\in\R^n$ and $B^T\sigma=0$. If $\Phi_D((\eps,\sigma))=0$, then $\sigma=\mathbb C \eps$, which implies $B^T\mathbb C Bu=0$. In turn, by \eqref{eqnetcomp}  this implies $u=0$, hence $(\eps,\sigma)=0$. Therefore, the restriction of $\Phi_D$ to $E_0$ is a strictly positive-definite quadratic form, and there is $c_1>0$ such that
$c_1\|\eta\|^2\le \Phi_D(\eta)$ for all $\eta\in E_0$. We now pick $e_0\in E\cap E_0^\perp$, write $z=\eta+e_0$ with $\eta\in E_0$, and use again that $\Phi_D$ is a quadratic form to write
\begin{equation}
 \Phi_D(\eta+e_0)=\Phi_D(e_0)+\Phi_D(\eta)+ D\Phi_D(e_0)(\eta) \ge c_1 \|\eta\|^2 - c_2\|\eta\|\ge \frac12 c_1\eta^2 -\frac{c_2^2}{2c_1},
\end{equation}
where $c_2$ is the operator norm of $D\Phi_D(e_0)$.
As $\|z\|^2=\|e_0\|^2+\|\eta\|^2$, this proves 
\eqref{p38xEk} with $c'=2/c_1$ and $b'=\|e_0\|^2+\frac{c_2^2}{c_1^2}$.

Similarly, there is $c''>0$ such that for any $y$ and $z$ we have
\begin{equation}\label{eqPhiD}
 \Phi_D(z)\le 2\Phi_D(z-y)+2\Phi_D(y)
 \le c'' \|z-y\|^2+2\Phi_D(y).
\end{equation}
Using 
$\|y\|^2+\|z\|^2 \le 3\|z\|^2+2\|y-z\|^2$,
then from \eqref{p38xEk} and \eqref{eqPhiD},
we obtain
\begin{equation}
\|y\|^2+\|z\|^2 \le 3c' \Phi_D(z)+3b'+2\|y-z\|^2
\le 
6c'\Phi_D(y)+3b'+(2+3c'c'')\|y-z\|^2.
\end{equation}
Therefore, there are $\beta_0,c,b>0$ such that
\begin{equation}
    c
    \big(
        \| y \|^2 + \| z \|^2
    \big)
\le
        \beta_0 \| y - z \|^2
    +
    \Phi_D(y)+b,
    \quad
    \forall z \in E , y\in Z,
\end{equation}
which is condition \eqref{9z2bLs}.
By Prop.~\ref{p5qmD1} we obtain that the measures are uniformly bounded and uniformly tight, hence transversal.
 
In order to use Prop.~\ref{VkH97k}, we observe that $\Phi_D\in C^\infty$ and  the growth condition \eqref{ceBj5Z} holds with $u(\xi):=\|\xi\|$ and $\gamma=1$, since $\Phi_D$ is quadratic in its arguments. From (\ref{8ddOkTdm}), the diagonal concentration of $\mu$ is defined by the property that
\begin{equation}\label{MEE3bQ}
    \int_{Z\times Z} f(y,z) \, d\mu_\infty(y,z)
    =
    \int_E
        f(\xi,\xi)
        \, {\rm e}^{-\Phi_D(\xi)}
    \, d\mathcal{H}^N(\xi) ,
\end{equation}
for all $f \in C_b(Z\times Z)$,
with rate of convergence as in (\ref{Kgx9hb}).
The total variation of $\mu_\infty$ 
is positive and finite in view of (\ref{p38xEk}) and, therefore, the expectation of a quantity of interest $f \in C_b(Z\times Z)$ follows by normalization as
\begin{equation}\label{eqexpectation}
    \mathbb{E}_\infty(f) = 
    \frac{1}{\|\mu_\infty\|_{\mathrm {TV}}}
    \int_{Z\times Z} f(y,z) \, d\mu_\infty(y,z)\,.
\end{equation}

\subsection{Approximation by empirical data}
We finally approximate  $\mu_D$ by discrete measures $\mu_{D,h}$ of the form (\ref{muDh}), and then apply Prop.~\ref{D7NtRs} and Cor.~\ref{y1TP50}. For every $h$ we define below a partition $\mathcal{A}_{e,h} = \{A_{\xip_e} \, : \, \xip_e \in P_{e,h} \}$ of $Z_e=\mathbb{R}^2$
on scale $\epsilon_{e,h}>0$. The  global material data set is then $P_h := \prod_{e=1}^N P_{e,h}$ and the partition of $Z$ is $\mathcal{A}_{h} := \{A_{\xip} = \prod_{e=1}^N A_{\xip_e} \, : \, \xip_e \in P_{e,h} \}$.  Setting $\epsilon_h:=(\sum_e \epsilon_{e,h}^2)^{1/2}$, we proceed to check Assumption~\ref{itbdtightcst2} and Assumption ~\ref{itbdtightcst} elementwise. Specifically, we verify \ref{itbdtightcst} for all $z\in Z$, not only for $z\in E$, since $E$ cannot be characterized elementwise.

We focus on a single element $e$ in the network and define $T_e:Z_e\to\R^2$, 
\begin{equation}
 T_e(\varepsilon_e,\sigma_e):=
\left(\frac{\varepsilon_e\sqrt{\C_e}}{\sqrt2}+
 \frac{\sigma_e}{\sqrt2\sqrt{\C_e}},
\frac{\varepsilon_e\sqrt{\C_e}}{\sqrt2}-
 \frac{\sigma_e}{\sqrt2\sqrt{\C_e}}\right),
 \end{equation}
so that 
$|T_e(y)|=\|y\|_e$,
with $|\cdot|$ the standard Euclidean norm
and $\|\cdot\|_e$ defined implicitly in \eqref{zQpQFw}.
We set $P_{e,h}:=\{p_e: T_e(p_e)\in \epsilon_{e,h}\Z^2\}$. For the rest of the construction, we
drop $h$ and $e$ from the notation.
Set
$A_{p}:=\{y: T(y)\in T(p)+(-\frac12\epsilon,\frac12\epsilon]^2\}$, which
defines a partition of $\R^2$ which obeys $p\in A_{p}$. Furthermore,
$\|y-p\|=|T(y-p)|\le \epsilon$ for all $y\in A_p$, so that
Assumption \ref{itbdtightcst2} follows. Assumption  \ref{itbdtightcst} holds with $c_*=\sqrt2$, since
$\|p-z\|^2\le 2
\|y-p\|^2+ 2
\|y-z\|^2$.

Assumption \ref{itbdtightub2} of Prop.~\ref{D7NtRs} follows directly from Prop.~\ref{p5qmD1}, since $\mu_D$ is sub-Gaussian. 
Setting $m_{p}:=\mu_{D}(A_{p})$ we see that
Assumption~\ref{itbdtightdelta} and
\eqref{condvi} hold, with $\bar c=1$ and $\delta_h=0$.
Finally, we choose an annealing schedule $\beta_h$ such that $\beta_h \epsilon_h^2 \to 0$, ensuring that Assumption \ref{itbdtighte2b} holds. Specifically, we may choose $\beta_h \sim \epsilon_h^{-1}$, as suggested by (\ref{eqintrooptbeta}) to obtain an approximation with the claimed quantitative convergence rate.

\section*{Acknowledgments}

This work was funded by the Deutsche Forschungsgemeinschaft (DFG, German Research Foundation) {\sl via} project 211504053 - SFB 1060; project 441211072 - SPP 2256; and project 390685813 -  GZ 2047/1 - HCM.

 \bibliography{biblio}
 \bibliographystyle{siamplain}

\end{document}